\documentclass[12pt, a4paper]{amsart}

\usepackage{a4wide}
\usepackage[english]{babel}
\usepackage[T2A]{fontenc}
\usepackage[utf8]{inputenc} 
\usepackage{amsfonts}
\usepackage{amssymb, amsthm, amscd}
\usepackage{amsmath}
\usepackage{mathtools}
\usepackage{needspace}
\usepackage{etoolbox}
\usepackage{lipsum}
\usepackage{comment}
\usepackage{cmap}
\usepackage[pdftex]{graphicx}
\usepackage[unicode]{hyperref}
\usepackage[matrix,arrow,curve]{xy}
\usepackage[usenames,dvipsnames]{xcolor}
\usepackage{colortbl}
\usepackage{textcomp}
\usepackage{cite}
\usepackage{euscript}

\makeatletter
\def\@settitle{\begin{center}
		\baselineskip14\p@\relax
		\bfseries
		\LARGE
		\@title
	\end{center}
}
\makeatother

\newtheorem{theorem}{Theorem}[section]
\newtheorem{lemma}[theorem]{Lemma}

\newtheorem{corollary}[theorem]{Corollary}
\theoremstyle{definition}
\newtheorem{definition}[theorem]{Definition}
\newtheorem*{remark}{Remark}

\renewcommand{\leq}{\leqslant}

\newcommand{\ot}{\otimes}
\newcommand{\x}{\otimes}

\DeclareMathOperator{\Ext}{Ext}
\DeclareMathOperator{\Hom}{Hom}

\title{$BV$-structure on Hochschild cohomology for exceptional local algebras of quaternion type. Case of the small parameter}
\author{Andrei V. Semenov}

\thanks{This work is supported by Young Russian Mathematics award and author is grateful to its jury and sponsors; is supported by ``Native Towns'', a social investment program of PJSC ``Gazprom Neft''. Also author is supported in part by The Euler International Mathematical Institute, grant number is 075-15-2019-16-20}

\address{Andrei V. Semenov:
Chebyshev Laboratory, 
St. Petersburg State University, 14th Line V.O., 29B, 
Saint Petersburg 199178 Russia}

\email{asemenov.spb.56@gmail.com}


\begin{document}
\maketitle

\begin{abstract}  This is the second paper in the cycle of articles about $BV$-structure on Hochshild cohomology of exceptional algebras of quaternion type. We give $BV$-structure's full description in the case of quaternion algebras $R(k,0,d)$, defined by parameter $k = 2$ according to Erdmann classification.

\end{abstract}
\vspace{10mm}

\section{Introduction}
A Hochschild cohomology is a well-known mathematical structure (see \cite{9}) and it is a large amount of papers studying it: see \cite{2, 5, 7, 8, 11, 18}, including author's paper \cite{6}, where one could find full list of articles about Hochschild cohomology. Tradler was first who introduced and developed $BV$-algebra structure on the Hochschild cohomology $HH^*(R)$, where $R$ is a finite dimensional symmetric algebra (see \cite{16}). According to his works, it is possible to define Gerstenhaber bracket (and a Lie algebra structure) on $HH^*(R)$ (see \cite{4} for example). There are plenty of papers developing the notion of $BV$-structure: see articles of Menichi \cite{14, 15}, Yang \cite{19}, Tradler \cite{16} and Ivanov \cite{12}, as well as Volkov's paper about $BV$-structure of Frobenius algebras (see. \cite{17}). \par
There is one big issue for computations: $BV$-structure is described in terms of bar-resolution, which makes it almost impossible to compute such structure for concrete examples, because the dimension of members of resolution grows exponentially. In order to avoid this we are using the method of comparsion morphisms for resolutions.\par 
This article stands as a second paper in a cycle of articles developing $BV$-structure on Hochschild cohomology for algebras of quaternion type (according to Erdmann's classification in \cite{3}). We consider a family of algebras $\{R(k,0,d)\}_{d \in K}$ over an algebraically closed field of characteristic 2 in case of $k=2$. It should be mentioned that one crucial partial case of $R(2,0,0)$ was calculated in \cite{10}. So our main Theorem generalizes the result of \cite{10}. One should note that the case of $BV$-structure on Hochschild cohomology for algebras $R(k,0,d)$ for even parameter $k>2$ was studied in the first paper of this cycle: see \cite{20}.\par

\section{Main definitions}
\subsection{Hochschild (co)homology}
Consider an associative algebra $A$ over the field $K$ of characteristic 2. Define its universal enveloping algebra as $A^e= A \otimes A^{op}$. Also we consider the free bar-resolution of $A$
$$\CD A @<{\mu}<< A^{\otimes 2} @<{d_1}<<  A^{\otimes 3} @<{d_2}<< ... @<{d_n}<< A^{\otimes n+2} @<{d_{n+1}}<< A^{\otimes n+3} ...\endCD$$
with differentials defined by rule
$$d_n(a_0 \x \dots a_{n+1}) = \sum \limits_{i=0}^{n} (-1)^i a_0 \x \dots \x a_i a_{i+1} \x \dots  \x a_{n+1}.$$
One can construct the {\it normalized bar-resolution} by setting $\overline{Bar}(A)_n = A \otimes \overline{A}^{\x n} \x A$, where $\overline{A} = A/ \langle 1_A \rangle$, and the differentials are induced by that of the bar-resolution.
\begin{definition} The $n$-th Hochschild {\it cohomology} is a space $HH^n(A) = \Ext^n_{A^e} (A,A)$ for any natural $n \ge 0$.
\end{definition}

\begin{definition} The $n$-th Hochschild {\it homology} $HH_n(A)$ is an $n$-th homology of the complex $A \x_{A^e} Bar_{\bullet}(A) \simeq A^{\bullet+1}$, where map $\partial_{n+1} : A^{\x (n+1)} \longrightarrow A^{\x n}$ sends $a_0 \x \dots \x a_n$ to $\sum_{i=0}^{n-1} (-1)^i a_0 \x \dots \x a_{i} a_{i+1} \x \dots \x a_n + (-1)^{n} a_na_0 \x \dots \x a_{n-1}$ and the differential is induced by that map. 

\end{definition}

One can define a cup-product on the Hochschild cohomology. For any classes $a \in HH^n(A)$ and $b \in HH^m(A)$ we define its cup-product $a \smile b \in HH^{n+m}(A)$ as the class of cup-product of their representatives $a \in \Hom_k(A^{\x n}, A)$ and $b \in \Hom_k(A^{\x m}, A)$. So by linear extension 
$$\smile : HH^n(A) \times HH^m(A) \longrightarrow HH^{n+m}(A)$$
the Hochschild cohomology space $HH^{\bullet}(A) = \bigoplus \limits_{n \ge 0} HH^n(A)$ becomes a graded-commutative algebra. 

\subsection{Gerstenhaber bracket}
For any $f \in \Hom_k(A^{\x n}, A)$ and $g \in \Hom_k(A^{\x m}, A)$ define $f \circ_i g \in \Hom_k(A^{\x n+m-1}, A)$ with the following rules:
\begin{enumerate}
    \item if $n \ge 1$ and $m \ge 1$, then let $f \circ_i g (a_1 \x \dots \x a_{n+m-1})$ be the formula $f(a_1 \x \dots a_{i-1} \x g(a_i \x \dots a_{i+m-1}) \x \dots \x a_{n+m-1})$.
    \item if $n \ge 1$ and $m=0$, then let $f \circ_i g (a_1 \x \dots \x a_{n-1})$ be the formula $f(a_1 \x \dots a_{i-1} \x g \x a_i \x  \dots \x a_{n+m-1})$. 
    \item otherwise let $f \circ_i g$ be equal to zero.
\end{enumerate}
\begin{definition}
Define {\it Gerstenhaber bracket} of $f \in \Hom_k(A^{\x n}, A)$and $g \in \Hom_k(A^{\x m}, A)$ by the formula
\[[f,g] = f \circ g - (-1)^{(n-1)(m-1)} g \circ f, \]
where $a \circ b = \sum \limits_{i=1}^n (-1)^{(m-1)(i-1)} a \circ_i b$.
\end{definition}
Obviously we have $[f,g] \in \Hom_k(A^{\x n+m-1}, A)$. Now for any $a \in HH^n(A)$ and $b \in HH^m(A)$ one can define $[a,b] \in HH^{n+m-1}(A)$ as a class of Gerstenhaber bracket of representatives of $a$ and $b$. So there exists a correctly defined map
$$ [-,-]:HH^{*} (A) \times HH^{*} (A) \longrightarrow HH^{*} (A),$$
which defines a structure of a graded Lie algebra on the $HH^*(R)$. This map will also be called {\it Gerstenhaber bracket} and it is not hard to see that $(HH^{*} (A), \smile, [-,-])$ is a Gerstenhaber algebra (see. \cite{4}).
\subsection{$BV$-structure}
\begin{definition}
By the {\it Batalin-Vilkovisky algebra} (or $BV$-algebra for short) we call Gerstenhaber algebra $(A^{\bullet}, \smile, [-,-])$ together with an operator $\Delta^{\bullet}$ of degree $-1$, such that $\Delta \circ \Delta = 0$ and
$$[a,b] = - (-1)^{(|a|-1)|b|} (\Delta(a \smile b) - \Delta(a) \smile b - (-1)^{|a|} a \smile \Delta(b)$$
for all homogeneous $a, b \in A^{\bullet}$.
\end{definition}
For $a_0 \x \dots \x a_n \in A^{\x (n+1)}$ define $\mathfrak{B}(a_0 \x \dots \x a_n)$ by the formula
\[\mathfrak{B}(a_0 \x \dots \x a_n)= \sum_{i=0}^n (-1)^{in} 1 \x a_i \x \dots \x a_n \x a_0 \x \dots \x a_{i-1} +\] 
\[ +  \sum_{i=0}^n (-1)^{in} a_i \x 1 \x a_{i+1} \x \dots \x a_n \x a_0 \x \dots \x a_{i-1}. \]
It is easy to see that $\mathfrak{B}(a_0 \x \dots \x a_n) \in A^{\x (n+2)} \simeq A \x_{A^e} A^{\x (n+2)}$, so it can be lifted to the chain complex map such that $\mathfrak{B} \circ \mathfrak{B} = 0$. So this is a correctly defined map on Hochschild homology.
\begin{definition}
The map $\mathfrak{B}: HH_{\bullet}(A) \longrightarrow HH_{\bullet+1}(A)$ is said to be Connes' $\mathfrak{B}$-operator.
\end{definition}
\begin{definition}
Algebra $A$ is called symmetric algebra, if it is isomorphic (as $A^e$-module) to its dual $DA = \Hom_K(A,K)$. 
\end{definition}
For a symmetric algebra $A$ one can always find the non-degenerate symmetric bilinear form $\langle -,- \rangle : A \times A \longrightarrow K$, and, obviously, reversed statement holds: for any such form the algebra $A$ is symmetric. So in case of symmetric algebras Hochschild homology and colomology are dual:
$$\Hom_K(A \x_{A^e} Bar_{\bullet}(A),K) \simeq \Hom_{A^e} (Bar_{\bullet}(A), \Hom_K(A,K)) \simeq \Hom_{A^e}(Bar_{\bullet}(A), A),$$
and one can define $\Delta: HH^n(A) \longrightarrow HH^{n-1}(A)$ as a dual to Connes' $\mathfrak{B}$-operator. \par 
Hence Hochschild cohomology of a symmetric algebra $A$ is a $BV$-algebra (see \cite{16}), and Connes' $\mathfrak{B}$-operator on homoplogy corresponds to $\Delta$ on cohomology. 
\begin{theorem}[Theorem 1, see \cite{13}]
Defined above cup-product, Gerstenhaber bracket and operator $\Delta$ induces a structure of $BV$-algebra on  $HH^{*}(A)$. Moreover, for any $f \in \Hom_K(A^{\x n}, A)$ the element $\Delta(f) \in \Hom_K(A^{\x (n-1)}, A)$ defined properly by formula
$$\langle \Delta(f)(a_1 \x \dots \x a_{n-1}),a_n \rangle = \sum_{i=1}^n (-1)^{i(n-1)} \langle f(a_i \x \dots \x a_{n-1}\x a_n \x a_1 \x \dots \x a_{i-1}), 1 \rangle$$
for any $a_i \in A$.
\end{theorem}

\begin{remark}
All constructions here can be defined and used in terms of the normalized bar-resolution.
\end{remark}

\section{Weak self-homotopy}
\subsection{Resolution}
Let $K$ be an algebraically closed field of characteristic 2 and fix $d \in K$. Consider an algebra $R(2, 0, d) = K \langle X,Y \rangle /I$, where $I = \langle X^2+YXY, Y^2+XYX + d(XY)^2, X(YX)^2, Y(XY)^2\rangle$ is an ideal in $K \langle X,Y \rangle$ (ans so one could easily check that $(XY)^2 + (YX)^2 \in I$). Let $B$ be the standard basis of the algebra $R=R(2, 0, d)$, so the set $B_1 = \{u \otimes v \mid u, v \in B\}$ is a basis for the enveloping algebra $\Lambda = R \otimes R^{op}$.\par
It should be mentioned that all algebras $R(2,0,d)$ are symmetric. This could be checked immediately by using such symmetric non-degenerate bilinear form:
$$\langle b_1, b_2\rangle = \begin{cases}
 1, & b_1 b_2 \in Soc(R) \\
 0, & \text{otherwise.}
\end{cases}$$
So in order to define graded Lie algebra structure on $HH^*(R)$ one only need to know how $\Delta$ (see definition 4) acts on the Hochschild cohomology. \par
Note that right multiplication by $\lambda \in \Lambda$ induces an endomorphism $\lambda^{*}$ of the left
$\Lambda$-module $\Lambda$; and we will denote it $\lambda$ as well. Sometimes we consider an endomorphism of the right $\Lambda$-module $\Lambda$, which is induced by left multiplication of $\lambda$. Let us denote such endomorphism as ${}^{*}\lambda$. \par
Now construct $4$-periodical resolution in the category of (left) $\Lambda$-modules
$$
\CD
P_0 @<{d_0}<< P_1 @<{d_1}<<  P_2 @<{d_2}<< P_3  @<{d_3}<< P_4 @<{d_4}<< \dots\\
\endCD$$
where $P_0=P_3=\Lambda$, $P_1 = P_2 = \Lambda^2$, and differentials defined by the formulae
$$d_0 = \begin{pmatrix}
x \otimes 1 + 1\otimes x & y \otimes 1 + 1 \otimes y
\end{pmatrix}, $$
$$d_1=\begin{pmatrix}
x \otimes 1 + 1\otimes x + y \otimes y, & 1 \x yx + xy \x 1 + d \x yxy + dxy \x y \\
1 \x xy + yx \x 1, & y \otimes 1 + 1\otimes y + x \otimes x + d x \x xy + dxyx \x 1
\end{pmatrix},$$ 
$$d_2 = \begin{pmatrix}
x \otimes 1 + 1\otimes x \\
y \otimes 1 + 1\otimes y + dy\otimes y + 1 \otimes dxyx+d^2y\otimes xyx\\
\end{pmatrix}, \quad d_3 =  \lambda^{*},$$  
where $ \lambda = \sum \limits_{0}^{2}(xy)^i \otimes (xy)^{2-i} + yx \otimes yx +\sum \limits_{0}^{1}y(xy)^i \otimes x(yx)^{1-i} +\sum \limits_{0}^{1}x(yx)^i \otimes y(xy)^{1-i} + dxyx \otimes xyx .$

Consider the multiplication map $\mu:\Lambda \longrightarrow R$, such that $\mu(a \otimes b) = ab$.
\begin{theorem}[Proposition 3.1, \cite{6}]
The complex $P_{\bullet}$ equipped with the map $\mu$ forms the minimal $\Lambda$-projective resolution of $R$.
\end{theorem}

We can represent the resolution $P_{\bullet}$ using the path algebra of $R$. One could define modules $KQ_1 = \langle x,y\rangle$ and $KQ_1^* = \langle r_x,r_y\rangle$, where $r_x = x^2+yxy$ and $r_y = y^2 + xyx+d(xy)^2$. It is easy to see that
$$R \otimes KQ_1 \otimes R = R \otimes \langle x \rangle \otimes R \oplus R \otimes \langle y \rangle \otimes R \simeq R \otimes R^{op} \oplus R \otimes R^{op} = \Lambda \oplus \Lambda,$$ 
so one could obtain the resolution $\{ P_n \}_{n=0}^{+\infty}$ of bimodules
$$
\CD
R @<{\mu}<< R \otimes R @<{d_0}<< R\otimes KQ_1 \otimes R @<{d_1}<< R\otimes KQ_1^* \otimes R  @<{d_2}<< R\otimes R @<{d_3}<< \dots
\endCD,$$
where $P_{n+4} = P_n$ for $n \in \mathbb{N}$. The differentials are defined by the formulae: 
\begin{itemize}
    \item $d_0(1 \otimes x \otimes 1) = x \otimes 1 + 1 \otimes x$, $d_0(1 \otimes y \otimes 1) = y \otimes 1 + 1 \otimes y$;
    \item $d_1(1 \otimes r_x \otimes 1) = 1\otimes x \otimes x + x \otimes x \otimes 1 +  y \otimes x \otimes y + 1 \x y \x xy + yx \x y \x 1$, \\
    $d_1(1 \otimes r_y \otimes 1) = 1\otimes y \otimes y + y \otimes y \otimes 1 +x \otimes y \otimes x + dx \x y \x xy + dxyx \x y \x 1 + 1 \x x \x yx + xy \x x \x 1 + d xy \x x \x y + d \x x \x yxy$;
    \item $d_2(1\otimes 1) = x \otimes r_x \otimes 1 + 1\otimes r_x \otimes x+ y\otimes r_y \otimes 1 + 1\otimes r_y \otimes y + d y\otimes r_y \otimes y + d \otimes r_y \otimes xyx + d^2 y \otimes r_y \otimes xyx$;
    \item $d_3 = \rho \mu$, where $\rho (1) = \sum \limits_{b \in B} b^*\otimes b + d xyx \otimes xyx$.
\end{itemize}
\subsection{Construction}
\begin{definition}
For the complex 
$$\CD
0 @<{}<< N @<{d_0}<< Q_0 @<{d_1}<< Q_1  @<{d_2}<< Q_2 \dots
\endCD$$
define {\it the weak self-homotopy} to be a collection of $K$-homomorphisms $t_{n} : Q_n \longrightarrow Q_{n+1} $ together with $t_{0} : N \longrightarrow Q_0$ such that $t_{n}d_n + d_{n+1}t_{n+1} = id_{Q_n}$ for all $n \geqslant 0$ and $d_0t_{0} = id_N$.
\end{definition}

We need to construct a weak self-homotopy $\{t_i :P_i \longrightarrow R_{i+1} \}_{i \geqslant -1}$ (here $P_{-1} = R)$ for such projective resolution, as in \cite{1}. 
Consider the bimodule derivation $C : KQ \longrightarrow KQ \otimes KQ_1 \otimes KQ$ by sending the path $\alpha_1...\alpha_n$ to $\sum \limits_{i=1}^n \alpha_1 ... \alpha_{i-1} \otimes \alpha_i \otimes \alpha_{i+1} ... \alpha_n$ and consider the induced map $C: R \longrightarrow R \otimes KQ_1 \otimes R$. So one could define $t_{-1} (1) = 1\otimes 1$ and $t_0 (b \otimes 1) = C(b)$ for $b \in B$. Now construct $t_1 : P_1 \longrightarrow P_2$ by the following rules: for $b\in B$ let
$$t_1(b \otimes x \ot 1) = $$
$$= \begin{cases} 
0, & bx \in B\setminus\{yxy\}\\
1 \ot r_x \ot 1, & b=x\\
1 \x r_x \x x^2 + x \x r_x \x x + x^2 \x r_x \x 1 +  yx \x r_y \x xy, & b = (xy)^2\\
T\Big(y \x r_x \x 1 + xy \x r_x \x y + 1 \x r_y \x xy +dx \x r_x \x xy  + dyxy \x r_x \x y \Big), & b = Tyx\\
y \x r_y \x 1+1 \x r_y \x y + dy \x r_y \x y + d\x r_y \x xyx+ d^2 y\x r_y \x xyx, & b = yxy,
\end{cases}
$$

$$t_1(b \otimes y \ot 1) = $$
$$ = \begin{cases} 
0, & by \in B \\

1 \ot r_y \ot 1, & b=y\\
T \Big(x \x r_y \x 1 + yx \x r_y \x x+ 1 \x r_x \x yx + d \x r_x \x yxy + \\
+ dyx \x r_y \x xy + dy \x r_x \x (xy)^2 \Big), & b = Txy.\\
\end{cases}
$$

In order to define $t_2 : P_2 \longrightarrow P_3$ use the following rules:
\begin{itemize}
    \item $t_2(x \x r_x \x 1) = 1 \x 1$,
    \item $t_2(y \x r_x \x 1) =  d yxy\x y^2 + d yx \x (xy)^2$,
    \item $t_2(xy \x r_x \x 1) = dyxy \x y + dyx \x y^2$,
    \item $t_2(yx \x r_x \x 1) = y\x 1 + d xyx \x 1 + dxy \x x + d x\x yx$,
    \item $t_2(yxy \x r_x \x 1) = 1 \x x$,
    \item $t_2 (xyx \x r_x \x 1) = xy \x 1 + x\x y + dyxy \x yx$,
    \item $t_2((xy)^2 \x r_x \x 1) = 1 \x yxy + yx \x y + y \x xy + yxy \x 1 +dxyx \x xy $,
\end{itemize}
and
\begin{itemize}
    \item $t_2(b \x r_y \x 1) = 0$ for $b \in \{x, y, xyx\}$,
    \item $t_2(xy \x r_y \x 1) = x\x 1 + dx \x y$,
    \item $t_2(yx \x r_y \x 1) =  dy^2 \x yxy + d(xy)^2 \x xy$,
    \item $t_2(yxy \x r_y \x 1) = yx \x 1 + y \x x + dyx \x y+ d y \x xy + d^2 xyx \x xy + d xyx \x x + dxy \x yxy$,
    \item $t_2((xy)^2 \x r_y \x 1) = xy \x x + x \x yx + xyx \x 1 + dx \x yxy + d xyx \x y+ dxy \x xy$.
\end{itemize}
Finally for $t_3 : R \x R \longrightarrow R \x R$ put $t_3((xy)^2 \x 1) = 1 \x 1$ and $t_3(b \x 1)=0$ otherwise. Now let $t_{n+4} = t_n$ for all $n \ge 4$.\par 
\begin{theorem}
The above-defined family of maps $\{t_i : P_{i} \longrightarrow P_{i+1}\}_{n \ge 0}$ together with $t_{-1}:R \longrightarrow P_0$ forms the weak self-homotopy for such resolution $P_{\bullet}$.
\end{theorem}
\begin{proof}
For any $n \in \mathbb{N}$ it remains to verify a commutativity of required diagrams, which is straight-up obvious from the definitions of $t_n$ for $n \leq 4$ and from the periodicity for $n \ge 5$.
\end{proof}

\section{Comparsion morphisms} 
Consider the normalized bar-resolution $ \overline{{Bar}}_{\bullet} (R) = R \otimes  \overline{{R}}^{\x \bullet } \otimes R$, where $ \overline{R} = R / (k\cdot 1_R)$. We now need to construct comparsion morphisms between $P_{\bullet}$ and $\overline{Bar}_{\bullet}(R)$:
$$\Phi : P_{\bullet} \longrightarrow  \overline{{Bar}}_{\bullet}(R) \text{ and } \Psi :  \overline{{Bar}}_{\bullet}(R) \longrightarrow  P_{\bullet}.$$
Note that there exists a weak self-homotopy $s_n(a_0 \x ... \x a_n \x 1) = 1 \x a_0 \x ... \x a_n \x 1$ of $ \overline{{Bar}}_{\bullet} (R)$, so put $\Phi_n = s_{n-1} \Phi_{n-1} d^P_{n-1}$ and $\Phi_0 = id_{R \x R}$.
\begin{lemma}
If $\Psi :  \overline{{Bar}}_{\bullet}(R) \longrightarrow  P_{\bullet}$ is the chain map constructed using $t_{\bullet}$, then for any $n \in \mathbb{N}$ and any $a_i \in R$ following formula holds
$$\Psi_n (1 \x a_1 \x ... \x a_n \x 1) = t_{n-1} (a_1 \Psi_{n-1} (1 \x a_2 \x ... \x a_n \x 1)).$$
\end{lemma}
\begin{proof}
The proof is an immediate consequence of Lemma 2.5 in \cite{10}.
\end{proof}

In order to define $BV$-structure on Hochschild cohomology one need to compute $\Delta : HH^n(R) \longrightarrow HH^{n-1}(R)$. By the Poisson rule we have
$$[a \smile b, c] = [a,c] \smile b + (-1)^{|a| (|c|-1)}(a \smile [b,c]),$$
and because char$K = 2$ it is easy to see that
$$\Delta(abc) = \Delta(ab)c + \Delta(ac)b + \Delta(bc)a + \Delta(a)bc +\Delta(b)ac + \Delta(c)ab.$$
So one need to compute $\Delta$ only on all generating elements and all cup-products of those elements. Furthermore, for any $\alpha \in HH^n(R)$ there exists a cocycle $f \in \Hom (P_n,R)$ such that following equality holds: 
$$\Delta(\alpha) = \Delta(f \Psi_n) \Phi_{n-1}.$$
Hence we have
$$\Delta(\alpha) (a_1 \x ... \x a_{n-1}) = \sum_{b \in B\setminus \{1\}} \sum \limits_{i=1}^n \langle (-1)^{i(n-1)} \alpha(a_i \x ... \x a_{n-1} \x b \x a_1 \x ...\x a_{i-1}), 1\rangle b^*,$$
where $\langle b,c \rangle$ is a defined-above bilinear form.

\section{$BV$-structure}

Let $K$ be an arbitrary algebraically closed field $K$ of characteristic two. Consider the set
$$\mathcal{X} = \{p_1, p_2,p_3,p_4,q_1,q_2,w_1,w_2,w_3,e\},$$ 
where the degrees of these elements are listed here:
$$
 |p_1| = |p_2| = |p_3| = |p_4| = 0, \ |q_1| = |q_2| = 1, \ |w_1| = |w_2| =|w_3|=2, \ |e|=4.
$$
Consider the ideal $\mathcal{I}$ in $K[\mathcal{X}]$, generated by such equalities:
\begin{itemize}
    \item degree 0: $p_ip_j$ for all $i,j \in \{1, 2, 3, 4\}$;
    \item degree 1: $p_3q_1+p_2q_2$, $p_1q_1+dp_3q_1 + p_3q_2$, $p_1q_2+p_2q_1$;
    \item degree 2: $p_2w_1$, $p_4w_1$, $p_3w_2$, $p_4w_2$, $p_4w_3$, $p_1w_1 + p_2w_2$, $p_1w_1 + p_3w_3$, $p_1w_1 + p_4q_1^2$, $p_3w_1+p_1w_2$, $p_3w_1 + p_2w_3$, $p_3w_1+p_4q_2^2$, $q_1q_2$, $p_1w_3+dp_2w_2$;
    \item degree 3: $q_1w_1+q_2w_2$, $q_1^3+q_2^3$, $q_1w_1 + q_2 w_1 + q_1w_3 + p_1q_1w_1$, $q_1w_1 + q_1w_2 + q_2 w_3 + p_1q_1w_1$;
    \item degree 4: $w_iw_j$ for all $i,j \in \{1, 2, 3\}$.
    \end{itemize}
\begin{theorem}[Theorem 2.1 in \cite{6}]
There exists a $K$-algebra isomorphism $HH^*(R) \simeq \mathcal{A} = K[\mathcal{X}]/\mathcal{I}$.
\end{theorem}
Let $P$ be an item of minimal projective resolution  $R$. If $P= R \x R$, then denote by $f$ the homomorphism in  $\Hom_{R^e}(P,R)$, which sends $1 \x 1$ to $f$. If $P= R \x KQ\x R$ (or $P= R \x KQ_1\x R$), then denote by $(f,g)$ the homomorphism in $\Hom_{R^e}(P,R)$, which sends $1\x x \x 1$ (or $1 \x r_x \x 1$) to $f$ and $1 \x y \x 1$ (or $1 \x r_y \x 1$) to $g$. So one can rewrite the generating elements like in \cite{6}
$$\begin{cases}
\text{elements of degree 0:} & p_1 = xy+yx, \ p_2 = xyx, \ p_3 = yxy, \ p_4 = (xy)^2, \\
\text{elements of degree 1:} & q_1 = q_1 = (y, \ 1+dy + xy),\ q_2 =(1 + yx, \ dxy +x), \\
\text{elements of degree 2:} & w_1 = ( x, \  0),\ w_2 = (0, \ y ),\ w_3= ( y,\ x + dxy ),\\
\text{elements of degree 4:} & e=1.
\end{cases}$$

\subsection{Technical lemmas}

It is obvious that $\Delta$ gives us zero on all elements of degree 0. 

\begin{lemma}[Elements of degree 1] We have
$\Delta(q_1) = dp_1$, $\Delta(q_2) = 0$, $\Delta(p_1q_1) = p_2+dp_1$, $\Delta(p_1q_2) = dp_2+p_3$, $\Delta(p_2q_1) = p_3+dp_2$, $\Delta(p_3q_2) = p_2$, $\Delta(p_3q_1) = \Delta(p_2q_2) = p_1$, $\Delta(p_4q_1) = p_2$, $\Delta(p_4q_2) = p_3$.
\end{lemma}
\begin{proof}
One need to compute $\langle a C(b), 1\rangle$ for all elements of degree 1. It is easy to see that
$$\langle a C(b), 1\rangle = 
\begin{cases}
1, & a \in \{ p_1q_1, p_3q_2, p_4q_1 \}, b = y \text{ or } \\ 
& a \in \{ p_1q_2, p_2q_1, p_4q_2 \}, b = x \text{ or } a \in \{p_3q_1, p_2q_2\}, b \in \{xy,yx\},\\
d, & a  \in \{q_1, p_1q_1\}, b \in \{xy, yx\} \text{ or } a \in \{p_1q_2, p_2q_1\},  b = y,\\
0, & \text{otherwise.}
\end{cases}$$
Hence we only need to note that $\Delta(a) = \sum \limits_{b \in B} \langle a C(b), 1\rangle b^*$, so the required statement holds.
\end{proof}

\begin{lemma}[Elements of degree 2]
For any combination $a \in HH^2(R)$ of generating elements we have $\Delta(a)=0$.
\end{lemma}
\begin{proof}
It is easy to see that for any such $a$ of degree two following formulae hold:
$$\Delta(a)(1 \x x \x 1) = \Delta(a \Psi_2) \Phi_1(1\x x \x 1) = \sum \limits_{b \ne 1} \langle a t_1 (b \x x \x 1 + x C(b)), 1\rangle b^*,$$
$$\Delta(a)(1 \x y \x 1) = \Delta(a \Psi_2) \Phi_1(1\x y \x 1) = \sum \limits_{b \ne 1} \langle a t_1 (b \x y \x 1 + y C(b)), 1\rangle b^*.$$
So we need to compute $t_1 (b \x x \x 1 + x C(b))$. Denote this formula by $\Psi_2(b,x)$:

\begin{enumerate}
\item $\Psi_2(xy,x) = 1 \x r_x \x y  + yx \x r_y \x 1 + y \x r_x \x yx + dy \x r_x \x yxy + dy^2 \x r_x \x (xy)^2$,
\item $\Psi_2(yx,x) = y \x r_x \x 1 + xy \x r_x \x y + 1 \x r_y \x xy + dx \x r_x \x xy + d yxy \x r_x \x y$,
\item $\Psi_2(xyx,x) = xy \x r_x \x 1 + x \x r_y \x xy  + dx^2 \x r_x \x xy + d(xy)^2 \x r_x \x y  + 1 \x r_x \x yx + yx \x r_y \x x + dy \x r_x \x (xy)^2$,
\item $\Psi_2(yxy,x)= y \x r_y \x 1 + 1\x r_y \x y + dy \x r_y \x y + d \x r_y \x xyx +d^2 y \x r_y \x xyx $,
\item $\Psi_2(b,x)=0$ for $b \in \{x, y, (xy)^2\}$.
\end{enumerate}
And also we need to compute $t_1 (b \x y \x 1 + y C(b))$. Denote this by $\Psi_2(b,y)$.
\begin{enumerate}
\item $\Psi_2(xy,y) = x \x r_y \x 1 + yx \x r_y \x x + 1 \x r_x \x yx + d\x r_x \x yxy + dyx\x r_y \x xy + dy \x r_x \x (xy)^2$,
\item $\Psi_2(yx,y) = 1 \x r_y \x x + xy \x r_x \x 1 + x \x r_y \x xy + dx^2 \x r_x \x xy + d(xy)^2 \x r_x \x y  + d \x r_x \x x^2 + dx \x r_x \x x + dx^2 \x r_x \x 1 + dyx \x r_y \x xy$,
\item $\Psi_2(xyx,y) =  y \x r_y \x 1 + 1\x r_y \x y + dy \x r_y \x y + d \x r_y \x xyx +d^2 y \x r_y \x xyx$,
\item $\Psi_2(yxy,y) = yx \x r_y \x 1 + y \x r_x \x yx + dy \x r_x \x yxy + dy^2 \x r_x \x (xy)^2 + 1 \x r_y \x xy + xy \x r_x \x y + d (xy)^2 \x r_x \x y^2 + d yxy \x r_x \x y $,
\item $\Psi_2((xy)^2,y) = y \x r_y \x y + 1 \x r_y \x xyx + dy \x r_y \x xyx$.
\item $\Psi_2(b,y) = 0$ for $b \in \{x, y\}$.
\end{enumerate}
Finally, one should note that
\[q_1q_2 = (0,\ 0), \ q_1^2 = (x, \ 1), \ q_2^2 = (1, \ y), \ p_4q_1^2 = (0, \ (xy)^2), \ p_4q_2^2 = ((xy)^2, \ 0), \] 
hence lemma holds by given computations.
\end{proof}

\begin{lemma}[Elements of degree 2] We have
$\Delta(q_1w_1) = \Delta(q_2w_2) = w_3, \ \Delta(q_2w_1) = q_2^2 + w_2, \ \Delta(q_1w_2) = \Delta(q_2w_3) = q_1^2 + w_1 + d(p_1+1)w_2, \ \Delta(q_1w_3) =  q_2^2 + w_2  + dw_3$.
\end{lemma}
\begin{proof} Using certain isomorphism $a_1 \x a_2 \x a_3 \equiv 1 \x a_1 \x a_2 \x a_3 \x 1$ one could show that 
$$\Delta(a)(1 \x r_x \x 1) = \Delta(a\Psi_3) \Phi_2 (1\x r_x \x 1) = \sum \limits_{b \not=1} \langle a\Psi_3(b \x x \x x + x \x x \x b + x \x b \x x), 1 \rangle b^* +$$
$$+ \sum \limits_{b \not=1} \langle a\Psi_3 (b \x y \x x + x \x b \x y + y\x x \x b), 1 \rangle b^* y +  \sum \limits_{b \not=1} \langle a\Psi_3 (b \x yx \x y + yx \x y \x b + y \x b \x yx), 1 \rangle b^* $$
for all $a \in HH^3(R)$. Note that $\Psi_3(a_1 \x a_2 \x a_3) = t_2 (a_1 t_1 (a_2 C(a_3)))$, so we can compute it directly.\par 
First of all, $\Psi_3 \big(b \x x \x x + x \x x \x b + x \x b \x x \big) = t_2 \big(b t_1 (x \x x \x 1) + xt_1 (b \x x \x 1) + xt_1(xC(b)) \big).$
Denote this formula by $\Psi_3(b,x)$. 
\begin{enumerate}
    \item $\Psi_3(x,x) = 1 \x 1$,
    \item $\Psi_3(y,x) = d yxy \x y^2 + dyx \x (xy)^2,$
    \item $\Psi_3(xy,x) = dyxy \x y + dyx \x y^2 +  1\x y$,
    \item $\Psi_3(yx,x) = y \x 1 + dxyx \x 1 + dxy \x x + dx \x yx$,
    \item $\Psi_3(xyx,x) = xy \x 1 + 1\x yx+ x \x y +  dyxy \x yx + yx \x xy + dyxy \x xy$,
    \item $\Psi_3(yxy,x) = 1 \x x + x\x 1$,
    \item $\Psi_3((xy)^2,x) = 1\x yxy$. 
\end{enumerate} \par 

Secondly, $\Psi_3 (b \x y \x x + y \x x \x b + x \x b \x y)$ could be rewritten as $t_2 (xt_1(b \x y \x 1) + yt_1(xC(b)))$. Denote this formula by $\Psi_3(b,x, 2)$. 
\begin{enumerate}
    \item $\Psi_3(x ,x, 2) = dyxy\x y^2 + dyx \x (xy)^2$,
    \item $\Psi_3(xy,x, 2) = yx \x 1 + y\x x + dyx \x y + dy\x xy + d^2 xyx \x xy + dxyx \x x + dxy \x yxy  + 1 \x yx + d\x yxy +  xy \x yx + dxy \x x^2 + dyxy \x (xy)^2+ d yxy \x yx +d^2yxy \x yxy$,
    \item $\Psi_3(xyx,x, 2) = dxy \x (xy)^2 + d^2 yxy \x (xy)^2$,
    \item $\Psi_3(b, x, 2) = 0$ otherwise. 
 \end{enumerate}   \par

Finally, $\Psi_3 (b \x yx \x y + yx \x y \x b + y \x b \x yx) = t_2 (yx t_1(yC(b)) + yt_1(b \x y \x x + by \x x \x 1))$. Denote this formula by $\Psi_3(b,x, 3)$. 
\begin{enumerate}
    \item $\Psi_3(x,x, 3) = 0$,
    \item $\Psi_3(y,x, 3) = d^2xyx\x x + d^2 xy \x x^2 +  dy^2 \x yxy + d (xy)^2 \x xy + 1 \x x + d (xy)^2 \x (xy)^2 + d y\x x$,
    \item $\Psi_3(xy,x, 3) = dy^2 \x (xy)^2 + d(xy)^2 \x xyx$,
    \item $\Psi_3(yx,x, 3) = dy^2 \x (xy)^2 + d(xy)^2 \x xyx + dy\x yxy + dxyx \x xy + dxy \x x + dx\x yx + dxyx\x 1 $,
    \item $\Psi_3(xyx,x, 3) = yx \x 1$,
    \item $\Psi_3(yxy,x, 3) = d(xy)^2 \x (xy)^2 + dxy\x (xy)^2+ d^2 yxy \x (xy)^2$,
    \item $\Psi_3((xy)^2,x, 3) = xy \x x^2 + xyx \x x + dx\x (xy)^2 + dxyx \x yx +dxy\x xyx$. 
  \end{enumerate}  \par  
Further, we have
 $$\Delta(a)(1 \x r_y \x 1) = \Delta(a\Psi_3) \Phi_2 (1\x r_x \x 1) = \sum \limits_{b \not=1} \langle a\Psi_3(b \x y \x y + y \x y \x b + y \x b \x y), 1 \rangle b^* +$$
$$ \sum \limits_{b \not=1} \langle a\Psi_3 (b \x x \x y + y \x b \x x + x \x y \x b), 1 \rangle b^* (x + d xy) + d\sum \limits_{b \not=1} \langle a\Psi_3 (b \x xyx \x y + xyx \x y \x b + y \x b \x xyx), 1 \rangle b^* $$$$+ \sum \limits_{b \not=1} \langle a\Psi_3 (b \x xy \x x + xy \x x \x b + x \x b \x xy), 1 \rangle b^* (1 + dy).$$

Note that $\Psi_3 (b \x y \x y + y \x y \x b + y \x b \x y) = t_2 (yt_1(b\x y \x 1) + yt_1(y C(b)) + b \x r_y \x 1)$. Denote this formula by $\Psi_3(b, y)$. So
\begin{enumerate}
    \item  $\Psi_3(b, y) = 0$ for $b \in \{x, y\}$,
    \item $\Psi_3(xy, y) = x\x 1+dx\x y  + dy^2 \x yxy + d(xy)^2 \x xy + d^2 yxy \x (xy)^2 + dxy \x (xy)^2 $,
    \item $\Psi_3(yx, y) = 1 \x x + d(xy)^2 \x (xy)^2 + d y \x x + d^2 xyx \x x + d^2 xy \x x^2  + dy^2 \x yxy + d(xy)^2 \x xy$,
    \item $\Psi_3(xyx, y) = dxy \x x + dx\x yx + dxyx \x 1$,
    \item $\Psi_3(yxy, y) = 1\x xy +yx\x 1 +y\x x + xy \x yx + dyxy \x yx + dxy \x x^2 + d^2 yxy \x yxy + dyx \x y + dxyx \x x + dxy \x yxy$,
    \item $\Psi_3((xy)^2, y) = xy \x x + x \x yx + xyx \x 1$.
  \end{enumerate}  \par  
  
Observe that $\Psi_3 (x \x y \x b + y \x b \x x + b \x x \x y)$ is equal to $t_2 (xt_1(yC(b)) +yt_1(b\x x \x 1))$. Denote this formula by $\Psi_3(b,y, 2)$. 
\begin{enumerate}
    \item $\Psi_3(x,y, 2)= dyxy \x y^2 + d yx \x (xy)^2$,
    \item $\Psi_3(yx,y, 2) = yx \x xy + dyxy \x xy + xy \x 1 + x \x y  + dyxy \x yx +  1 \x xy  +  d^2xyx \x xy + dy \x xy$,
    \item $\Psi_3(xyx,y, 2) = x\x 1+ 1\x x  + d(xy)^2 \x (xy)^2$,
    \item $\Psi_3(yxy,y, 2) = dyx \x y^2 + dyxy \x y +  dxy \x x + dx \x yx + dxyx\x 1 $,
    \item $\Psi_3((xy)^2,y, 2) = x \x y + y\x x + dxyx \x x + dxy \x x^2$, 
    \item $\Psi_3(b,y, 2) = 0$ for any $b \in \{y, xy \}$. 
  \end{enumerate}  \par

Rewrite $\Psi_3 (xyx \x y \x b + y \x b \x xyx + b \x xyx \x y)$ as $t_2 (xyx t_1(yC(b)) + yt_1(b \x x \x yx + bx \x y \x x + bxy \x x \x 1))$. Denote this formula by $\Psi_3(b,y, 3)$. 
\begin{enumerate}
    \item $\Psi_3(x,y, 3) = dxy \x (xy)^2 + d^2 yxy \x (xy)^2$,
    \item $\Psi_3(y,y, 3) = dxy \x x + dx\x yx + dxyx \x 1 $,
    \item $\Psi_3(xy,y, 3) = y\x x + dxyx \x x + dxy \x x^2$,
    \item $\Psi_3(yx,y, 3) = dxy \x yxy + xy \x yx + d yxy \x (xy)^2 + dyxy \x yx$,
    \item $\Psi_3(xyx,y,3) = xy \x x + x \x yx + xyx \x 1 + 1 \x xyx $,
    \item $\Psi_3(yxy,y, 3) = xy \x x^2 + xyx \x x$,
    \item $\Psi_3((xy)^2,y, 3) = xy \x xy + x \x yxy + xyx\x y + dxyx \x xyx +  y \x xyx $. 
  \end{enumerate}  \par
  
  Finally, denote  $\Psi_3 (xy \x x \x b + x \x b \x xy + b \x xy \x x) = t_2 (xy t_1 (xC(b)) + xt_1(b \x x \x y + bx \x y \x 1))$ by $\Psi_3(b,y, 4)$. 
\begin{enumerate}
    \item $\Psi_3(x,y, 4) = 1\x y +  dyxy \x y + dyx \x y^2$,
    \item $\Psi_3(xy,y, 4) = dyxy \x y^2 + dyx \x (xy)^2 $,
    \item $ \Psi_3(yxy,y, 4) = x \x y $,
    \item $\Psi_3((xy)^2,y, 4) = yx \x y^2 + yxy \x y $, 
   \item $\Psi_3(b,y, 4) = 0$ for $b \in \{y, yx, xyx \}$.
  \end{enumerate}  \par
  
We now only need to note that
  $$q_1w_1 = xy, \ q_1w_2 = y = q_2w_3,\  q_2w_1 = x, \ q_2w_2 = yx, \ q_1w_3 = x+dyx +d(xy)^2,$$ 
hence required formulae holds.

\end{proof}

\begin{lemma}[Степень 4] We have $\Delta(e) = \Delta(p_4 e)  = d^3 p_1 q_1 w_1$, $\Delta(p_1e) = \Delta(p_2e) = \Delta(p_3e) = 0$.
\end{lemma}
\begin{proof} By $circ(a_1 \x ...\x a_n)$ we denote $\sum a_{i_1} \x ... \x a_{i_n}$, where in the sum we are counting all indices such that $(i_1, \dots , i_n) = (1,2, \dots ,n)^{-1}$. So for any $a \in HH^4 (R)$ we have 
$$\Delta(a) (1 \x 1) = \Delta(a \Psi_4) \Phi_3 (1\x 1) =  \sum \limits_{b \not=1} \langle a \Psi_4 (circ(b \x x \x x \x x)), 1 \rangle b^*  $$
$$+ \sum \limits_{b \not=1} \langle a \Psi_4 (circ(b \x x \x y \x x)), 1 \rangle b^* y + \sum \limits_{b \not=1} \langle a \Psi_4 (circ(b \x x \x yx \x y)), 1 \rangle b^*  +$$
$$\sum \limits_{b \not=1} \langle a \Psi_4 (circ(b \x y \x y \x y)), 1 \rangle b^* (1 +dy +d^2 xyx) + \sum \limits_{b \not=1} \langle a \Psi_4 (circ(b \x y \x x \x y)), 1 \rangle b^* x +$$
$$\sum \limits_{b \not=1} \langle a \Psi_4 (circ(b \x y \x xy \x x)), 1 \rangle b^* + d\sum \limits_{b \not=1} \langle a \Psi_4 (circ(b \x y \x xyx \x y)), 1 \rangle b^* (1 +dy +d^2 xyx).$$
Denote by $\Psi_4(b,i)$ the value of $b$-th summand in the $i$-th sum from above. Direct computations show us that  
\[ \Psi_4(b, 1) = \begin{cases}
d \x y^2, & b = y\\
d \x y, & b = xy\\
d \x yx + d \x xy, & b = xyx \\ 
1 \x 1, & b = (xy)^2 \\
0, & \text{otherwise}
\end{cases} \ \text{ and } \ 
\Psi_4(b, 2) = \begin{cases}
d \x y^2, & b = x\\
d \x (xy)^2 + d \x yx + d^2 \x yxy, & b = xy\\
d^2 \x (xy)^2, & b = xyx \\ 
0, & \text{otherwise,}
\end{cases}
\]

\[
\Psi_4(b, 3) = \begin{cases}
d \x (xy)^2, & b = x\\
d \x y, & b = yx\\
0, & \text{otherwise}
\end{cases} \ \text{ and } \ 
\Psi_4(b, 4)+ \Psi_4(b, 7)= \begin{cases}
d^2 \x 1 , & b = y\\
d \x x^2, & b \in \{xy, yx\}\\
d \x 1, & b = xyx \\ 
1 \x 1, & b = (xy)^2 \\
0, & \text{otherwise,}
\end{cases}
\]
\[
\Psi_4(b,5) = \begin{cases}
d^2 \x (xy)^2, & b = xy\\
d \x yx + d^2 \x yxy, & b = yxy\\
0, & \text{otherwise}
\end{cases} \ \text{ and } \ 
\Psi_4(b,6) = \begin{cases}
d \x yxy, & b = xy\\
0, & \text{otherwise.}
\end{cases}
\]
So the required formulae now can be deduced from the above-given computations.
\end{proof}

\begin{remark}
It is useful to know the direct form of the map $\Phi_4$:
$$\Phi_4 (1 \x 1) = \sum_b 1 \x  b \x  x \x  x \x  x \x  b^* + \sum_b 1 \x  b \x  x \x  y \x  x \x  yb^*+ $$
$$\sum_b 1 \x  b \x  x \x  yx \x  y \x  b^* + \sum_b 1 \x  b \x  y \x  y \x  y \x (1+dy + d^2 xyx)b^* + \sum_b 1 \x  b \x  y \x  x \x y \x  xb^* + $$
$$ \sum_b 1 \x  b \x  y \x  xy \x  x \x  b^* + \sum_b 1 \x  b \x  y \x  xyx \x  y \x  (d+d^2y + d^3 xyx)b^* +$$
$$d \x  xyx \x  x \x  x \x  x \x  xyx+ d \x  xyx \x  x \x  y \x  x \x  (xy)^2 + d \x  xyx \x  x \x  yx \x  y \x  xyx +$$
$$d \x  xyx \x  y \x  y \x  y \x  y^2+d \x  xyx \x  y \x  xy \x  x \x  xyx + d^2 \x  xyx \x  y \x  xyx \x  y \x  y^2.$$

We denote by $\Phi_4^i$ the $i$-th summand from this expression for any $ 1 \leq i \leq 13$ (here we use the fixed order of summands as we write in the above presented formula).
\end{remark}

\subsection{Gerstenhaber brackets}
\begin{lemma}  We have $[q_1,e] =0$ and $[q_2, e] = dp_2e$.
\end{lemma}
\begin{proof}
It is not hard to show that for any $a \in HH^1(R), e \in HH^4(R)$ we have
$$[a,e] (1 \x 1) = (a \Psi_1 \circ e\Psi_4) \Phi_4 (1\x 1) + (e \Psi_4 \circ a \Psi_1) \Phi_4 (1\x 1).$$
Now observe that $\Phi_4 (1\x 1) = \sum \limits_{b \in B} 1 \x b \Phi_3 (1 \x 1) b^* + d \x xyx \Phi_3 (1 \x 1)xyx$. If we apply $d_3^{Bar}$ to $1 \x b \Phi_3 (1\x 1) b^*$, and then apply $t_3 \Psi_3$ to the resulting formula, then the only non-zero summand will be $b \Phi_3 (1 \x 1) b^*$: indeed, $t_3( 1 \cdot \Psi_3(s)) = t_3t_2 (s) = 0$ for any $s$ from the domain of $\Psi_3$, so required equality holds. Hence we have
$$\Psi_4 \Phi_4 = t_3 \Psi_3 d_3^{Bar} \Phi_4 = \sum_b t_3 (b \Psi_3 \Phi_3 )b^*+ dt_3 (xyx \Psi_3 \Phi_3 )xyx,$$
and $\Psi_3 \Phi_3 (1 \x 1) = t_2 (xt_1(x \cdot x \cdot 1)) + t_2(yt_1( y \cdot y \cdot 1))(1 + dy +d^2 xyx) = t_2 (x \cdot r_x \cdot 1) = 1 \x 1$, so
$$(a \Psi_1 \circ e\Psi_4) \Phi_4 (1\x 1) = (a \Psi_1) (e \Psi_4 \Phi_4 (1\x 1)) = (a \Psi_1) (1) = 0.$$
Now we only need to calculate the second summand from the expression given above for $[a,e] (1 \x 1)$. By definition of the function $e \Psi_4 \circ a \Psi_1$ it is equal to the sum of four functions $F_i^a$ for any generating $a \in HH^1(R)$ and any $1 \leq i \leq 4$. For the calculations one need to know the values of $a \Psi_1 (b)$ for any $b \in B$:
$$q_1 \Psi_1 (b) =  \begin{cases}
y, & b = x \\
1 + xy + dy, & b = y \\
x + dxy + y^2 , & b = xy \\
x + dyx + d(xy)^2, & b = yx \\
yxy+ dxyx, & b = xyx \\
xy + yx, & b = yxy \\
xyx, & b = (xy)^2, 
\end{cases}
\quad
q_2 \Psi_1 (b) =  \begin{cases}
1 + yx, & b = x \\
d xy + x, & b = y \\
y, & b = xy \\
y +  x^2 + dxyx, & b = yx \\
xy + yx, & b = xyx \\
xyx, & b = yxy\\
yxy, & b = (xy)^2.
\end{cases}
$$

It is easy to see that all summands from $\Phi_4 (1\x 1)$ have the form $1 \x a_1 \x \dots \x a_5 \x a_6$, and if $a_4a_5 \in B$, then these summands gives zero after applying $F_1^a$ or $F_2^a$ to them. \par 

1) Denote by $f^a_i(b)$ the value of $(e \Psi_4 \circ_1 q_a\Psi_1) (1 \x b \x a_1 \x a_2 \x a_3 \x a_4)$, where $1 \x b \x a_1 \x a_2 \x a_3 \x a_4$ is a summand in $\Phi_4^i$. So

$$f^1_1 (b) = \begin{cases}
d xy, & b = xy\\
dyx, & b= yx\\
0, & \text{otherwise}
\end{cases} \text{ and } \quad 
f^2_1 (b)  = 0. $$
Now observe that
$$f^2_{8} (xyx) = d \cdot e t_3 (a \Psi_1 (xyx) \x 1)  xyx = 0 \text{ and } f^1_{8} (xyx)  =0,$$
so 
$$F_1^{q_1} = d(xy+yx) \text{ and } F_1^{q_2} =0.$$

2) By $g^a_i(b)$ we denote the value of $(e \Psi_4 \circ_2 q_a \Psi_1) (1 \x b \x a_1 \x a_2 \x a_3 \x a_4)$, where $1 \x b \x a_1 \x a_2 \x a_3 \x a_4$ is a summand of $\Phi_4^i$. We obtain
$$g_1^1 (b) =  0, \quad g_1^2 (b) =\begin{cases}
y, & b = xyx \\
0, & \text{otherwise,}
\end{cases} \quad 
g_4^1 (b) = \begin{cases}
x , & b = yxy\\
0, & \text{otherwise,}
\end{cases}$$
and so $g^2_{8} (xyx)  = dxyx$, $g^1_{8} (xyx) =0$. It is not hard to show that $\sum_b g_4^2 (b) = d \sum_b g_4^1(b) + \sum_b et_3 (b t_2 (x\x r_y \x 1))(1+dy+d^2xyx)b^*$. Now observe that the summands of the second sum gives us zero for all $b \in B$. Finally, one has $g_{11}^a(xyx) = 0$ for any generating $a \in HH^1(R)$. To sum up, we can conclude that
$$F_2^{q_1} = x \text{ and } F_2^{q_2} =  y + dx + d xyx.$$

3) Denote by $h^a_i(b)$ the value of $(e \Psi_4 \circ_3 q_a \Psi_1) (1 \x b \x a_1 \x a_2 \x a_3 \x a_4)$, where $1 \x b \x a_1 \x a_2 \x a_3 \x a_4$ is a summand of $\Phi_4^i$. Then we have $h_1^2(b) = 0$ and $h_1^1(b) = 0$ for any $ b \in B$ by definitions of $t_i$. Hence $h_{8}^1(xyx) = h_{8}^2(xyx) = 0$. So,
$$h_2^a(b) = 
\begin{cases} y, & b = (xy)^2 \text{ and } a= q_2\\
0, & \text{otherwise}
\end{cases} \text{  and  } h_3^a(b) = 0 \text{ for any } a \in HH^4(R),$$ 
hence $h_3^a(b) = h_{13}^a(b) = 0$ for any $a$. It is easy to see that $h_4^1(b) = 0$ and $h_4^2(b) = 0$ for any $b$ because $t_1(x \x y \x 1) = 0$. Analogically we have $ h_{11}^a (b) = 0$ for any $a$. Moreover, since $t_2(y t_1 (a\Psi_1 (x) \x y \x 1)) = 0$ we have $h_5^a (b) = 0$. Furthermore,
$$h_7^1(b) = \begin{cases}
d xy, & b = xyx\\
x , & b = (xy)^2\\
0, & \text{ otherwise}
\end{cases} \text{ and } h_7^2 (b) = 0.$$
It is easy to see that $h^a_{12} (xyx) = 0$ and $h_7^a(b) = 0$ for any $b \in B$ and any $a \in HH^1(R)$, so we have $ h_{13}^a(xyx) = 0$. Hence
$$F_3^{q_1} = x + dxy \text{ and } F_3^{q_2} = y.$$

4) It remains to describe the $F_4^a$ for any $a \in \{1, 2\}$. Denote by $k^a_i(b)$ the value of $(e \Psi_4 \circ_4 q_a \Psi_1) (1 \x b \x a_1 \x a_2 \x a_3 \x a_4)$, where $1 \x b \x a_1 \x a_2 \x a_3 \x a_4$ is a summand of $\Phi_4^i$. For the calculation we need the following formulae:
$$C(a \Psi_1 (x)) = \begin{cases}
1 \x y \x 1, & a = q_1\\
y \x x \x 1 + 1 \x y \x x, & a = q_2,
\end{cases}$$ 
$$C(a \Psi_1 (y)) = \begin{cases}
1 \x x \x y + x \x y \x 1 + d\x y \x 1, & a = q_1\\
d \x x \x y + dx \x y \x 1 + 1 \x x \x 1, & a = q_2.
\end{cases}$$ 
It follows from the definitions of $t_i$ and $\Phi_4$ that $k_7^1((xy)^2) = dxy$, $k^2_7((xy)^2)= dx$ and $k_i^j(b) = 0$ for any $(i,b) \not= (7,(xy)^2)$ and any $j \in \{1,2\}$. So we have
\[F_4^{q_1} = dxy \text{ and } F_4^{q_2} = \sum \limits_{b,i} k_i^2(b) = dx.
\]
It remains to compute the Gerstenhaber brackets:
\[[q_1,e] = (q_1 \Psi_1 \circ e\Psi_4) \Phi_4 (1\x 1) + (e \Psi_4 \circ q_1 \Psi_1) \Phi_4 (1\x 1) = \sum_{i=1}^4 F_i^{q_1} = dxy+dyx \equiv 0,\]
\[[q_2,e] =(q_2 \Psi_1 \circ e\Psi_4) \Phi_4 (1\x 1) + (e \Psi_4 \circ q_2 \Psi_1) \Phi_4 (1\x 1) = \sum_{i=1}^4 F_i^{q_2} = dxyx.\]

\end{proof}

\begin{lemma}
$[v,e] = 0$ for any $v \in \{w_1, w_2, w_3\}$.
\end{lemma}
\begin{proof}
For any $v \in HH^2(R)$ and any $e \in HH^4(R)$ we have
$$[v,e] (1 \x a \x 1) = \big( (v \Psi_2) \circ (e\Psi_4) \big) \Phi_5 (1 \x a \x 1) + \big( (e\Psi_4)\circ (v \Psi_2) \big) \Phi_5 (1 \x a \x 1).$$ 
We now need to calculate $\big( (v \Psi_2) \circ (e\Psi_4) \big) \Phi_5( 1 \x a \x 1)$:
\[ 
\big( (v \Psi_2) \circ (e\Psi_4) \big) \Phi_5( 1 \x a \x 1) = \sum_{i=1}^2 \big( (v \Psi_2) \circ_i (e\Psi_4) \big) \Phi_5( 1 \x a \x 1),
\]
and we denote the summands of this sum by $S_1^v$ and $S_2^v$ respectively. It is easy to see that $S_2^v = 0$ for any $v \in \{ w_1, w_2, w_3\}$: indeed we have
\[ S_2 (a_1 \x .. \x a_5 \x a_6) = v t_1 (a_1 C(et_3 (a_2 t_2 (a_3 t_1( a_4 C(a_5) ))))) \cdot a_6, \]
and this formula equals zero on all summands of the form $1 \x a_1 \dots a_5 \x a_6$ from the definition of $\Phi_5$ because $C(1) = 0$ and $et_3(b \x 1) =1$ for $ b = (xy)^k $ and equals zero elsewhere. Let us proof that $S_1^v$ gives us zero on all summands of $\Phi_5 (1\x a \x 1)$ except maybe first, fourth, seventh and eleventh summands:
\[
S_1^{v} (a \x xyx \x x \x x \x x \x y) = \begin{cases} dyxy, & a = x \text{ and } v = w_1 \\
d(xy)^2, & a = x \text{ and } v = w_3 \\ 
0, & \text{otherwise,} 
\end{cases}
\]
\[
S_1^{v} \big(a \x (xy)^2 \x y \x y \x y \x (1 +dy + d^2 xyx) \big) = \begin{cases} dy + dxyx, & a = y \text{ and } v = w_2 \\
dx, & a = y \text{ and } v = w_3 \\ 
0, & \text{otherwise,} 
\end{cases}
\]
\[
S_1^{v} \big(a \x (xy)^2 \x y \x xyx \x y \x (d +d^2y + d^3 xyx) \big) = \begin{cases} dy + dxyx, & a = y \text{ and } v = w_2 \\
dx, & a = y \text{ and } v = w_3 \\ 
0, & \text{otherwise,} 
\end{cases}
\]
and $S_1^v$ gives us zero on all other combinations of elements $a_i$. So we have
\[ \big( (v \Psi_2) \circ (e\Psi_4) \big) \Phi_5( 1 \x a \x 1) = \begin{cases} dyxy, & a = x \text{ and } v = w_1 \\
d(xy)^2, & a = x \text{ and } v = w_3 \\ 
0, & \text{otherwise.} 
\end{cases}
\]

Now we need to compute $\big( (e\Psi_4)\circ (v \Psi_2) \big) \Phi_5 (1 \x a \x 1) = \sum \limits_{i=1}^4 \big( (e\Psi_4)\circ_i (v \Psi_2) \big) \Phi_5 (1 \x a \x 1)$. We denote by $F^v_i$ the summands of this sum for any $1 \leq i \leq 4$. It is easy to see that $F^v_1$, $F^v_2$ and $F_4^v$ may {\it not} equal to zero only for combinations of elements from the first, fourth, seventh and eleventh summands of $\Phi_5 (1 \x a \x 1)$, because any other summand has the form $a_1 \dots a_5 \x a_6$, where $t_1 (a_4 C(a_5) ) = 0$. \par 
1) Consider the function $F_1^v$. Obviously $t_2 \big(y t_1 (y \x y \x 1) \big) = 0$, so $F_1^v$ equals zero on the fourth and eleventh summands. It remains to show that
\[ et_3 \Big(vt_1 \big(aC(b) \big) t_2 \big(x t_1 (x \x x \x 1)\big) \Big) b^* = et_3\big(vt_1 (aC(b)) \x 1 \big) b^* \]
and for the first summand

\[ F_1^v (a \x b \x x \x x \x x \x b^*) = et_3\big(vt_1 (aC(b)) \x 1 \big) b^* = \]
\[
 = \begin{cases}
xy, & a =x,  \text{ } b = xy \text{ and } v = w_1 \\
dxy, & a =x,  \text{ } b = xy \text{ and } v = w_3 \\
dyx, & a =y,  \text{ } b = yx \text{ and } v = w_1 \\
yx, & a =y, \text{ } b = yx \text{ and } v = w_2 \\
y, & a = x, \text{ } b = xyx \text{ and } v = w_2 \\
x, & a = y, \text{ } b = yxy \text{ and } v = w_1 \\
1, & b = (xy)^2 \text{ and } (a,v) = (x, w_1) \text{ or } (a,v) = (y, w_2)\\
0, & \text{otherwise.}
\end{cases}
\]

So for the eighth summand we have
\[ d F_1^v (a \x xyx \x x \x x \x x \x xyx) = 
\begin{cases}
dxyx, & a = x \text{ and } v = w_2 \\
0, & \text{otherwise.}
\end{cases}
\]
And from the definitions of $t_i$ we show that for any another combination $a_1 \x  \dots \x a_6$ from the summands of $\Phi_5 (1 \x a \x 1)$ gives us zero. So we have
\[ F_1^{w_1} = \begin{cases} 1+xy, & a =x \\x + dyx, & a = y, \end{cases} \quad F_1^{w_2} = \begin{cases} y+dxyx, & a = x \\ 1+yx, & a = y, \end{cases} \]
\[
 F_1^{w_3} = \begin{cases} dxy, & a = x \\ 0, & a = y. \end{cases}\]
 2) In the case of $F_2^v$ for the first summand we have
\[ F_2^{v}(1 \x a \x b \x x \x x \x x \x b^*) = \begin{cases} 
et_3\big(at_2 (yx \x r_x \x 1 + (xy)^2 \x r_x \x 1)\big)yx, & b = yx \text{ and } v = w_1 \\
et_3\big(at_2 (y^2 \x r_x \x 1)\big)yx, & b = yx \text{ and } v = w_3 \\
et_3\big(at_2 (xyx \x r_x \x 1)\big)y, & b = xyx \text{ and } v = w_1 \\
et_3\big(at_2 ((xy)^2 \x r_x \x 1)\big)y, & b = xyx \text{ and } v = w_2 \\
et_3\Big(at_2 \big((yx + xy + 2d yxy)\x r_x \x 1\big)\Big)x, & b = yxy \text{ and } v = w_3 \\
et_3\big(at_2 ((xy)^2\x r_x \x 1)\big), & b = (xy)^2 \text{ and } v = w_1 \\
et_3\big(at_2 (xyx\x r_x \x 1)\big), & b = (xy)^2 \text{ and } v = w_3 \\
0, & \text{otherwise,}
\end{cases} =
\]
\[ = \begin{cases} 
yx, & a= x, \text{ } b = yx \text{ and } v = w_1 \\
d(xy)^2 + dyx, & a= x, \text{ }  b = yx \text{ and } v = w_3 \\
dyxy, & a= x, \text{ } b = xyx \text{ and } v = w_1 \\
y, & a= x, \text{ } b = xyx \text{ and } v = w_2 \\
dyx, &a= x, \text{ } b = yxy \text{ and } v = w_3 \\
1, & a= x, \text{ }  b = (xy)^2 \text{ and } v = w_1 \\
dyx, & a= x, \text{ } b = (xy)^2 \text{ and } v = w_3 \\
0, & \text{otherwise.}
\end{cases}
\]
So for the eighth summand we obtain
 \[ d F_2^v(1 \x a \x xyx \x x \x x \x x \x xyx) = \begin{cases}
 dxyx, & a = x \text{ and } v = w_2\\
 0, & \text{otherwise.}
 \end{cases}
 \]
Now in case of fourth and eleventh summands 
 $$F_2^v\big(a \x b \x y \x y \x y \x (1+dy+d^2xyx)b^*\big) = et_3\Big( at_2 \big( vt_1(b \x y \x 1) \x r_y \x 1\big) \Big)(1+dy+d^2xyx)b^*,$$
 so
 \[F_2^v\big(a \x b \x y \x y \x y \x 1 \big) =\]
 \[= \begin{cases}
 et_3\big(at_2(xyx \x r_y \x 1 + d(xy)^2 \x r_y \x 1) \big), & b = xy \text{ and } v = w_1\\
 et_3\big(at_2(xy \x r_y \x 1 + (xy)^2 \x r_y \x 1) \big), & b = xy \text{ and } v = w_2\\
 et_3\big(at_2((xy)^2 \x r_y \x 1) \big), & (b, v) = (yxy, w_1) \text{ or } (b, v) = ((xy)^2, w_2) \\
 \end{cases}
 \]
and $F_2^v$ gives us zero for any other combinations of elements after right multiplying by $(1+dy+d^2xyx)b^*$. So we only need to check that
  \[F_2^v\big(a \x b \x y \x y \x y \x (1 +dy + d^2xyx)b^* \big) = \begin{cases}
dxy, & a=y, \text{ } b = xy \text{ and } v = w_1\\
xy, & a=y, \text{ } b = xy \text{ and } v = w_2\\
x, & a=y, \text{ } b = yxy \text{ and } v = w_1 \\
1, & a=y, \text{ } b = (xy)^2 \text{ and } v = w_2 \\
0, & \text{otherwise.}
 \end{cases}
 \]
So on the eleventh summand $F_2^v$ gives us zero and hence
 \[ F_2^{w_1} = \begin{cases} 1+yx + dyxy, & a =x \\x + dxy, & a = y, \end{cases} \quad F_2^{w_2} = \begin{cases} y+dxyx, & a = x \\ 1+xy, & a = y, \end{cases}\]
\[
 F_2^{w_3} = \begin{cases} dyx + d(xy)^2, & a = x \\ 0, & a = y. \end{cases}\]
 
 3) In order to calculate $F_3^v$ we need to note that if $t_1 \big(a_3 C(a_4) \big)=0$ for a summand of the form $a_1 \x \dots \x a_5 \x a_6$ then this summand gives us zero after applying $F_3^v$ to the summand. If $(a_3,a_4,a_5) = (x,x,x)$, then
 \[ t_1 \big(vt_1(x \x x \x 1) \x x \x 1 \big) = \begin{cases} 1\x r_x \x 1, & v = w_1 \\ 0, & v \not= w_1, \end{cases}\]
and so 
\[ F_3^{v} (1 \x a \x b \x x \x x \x x \x 1) = \begin{cases} 1, & a = x, \text{ } b = (xy)^2 \text{ and } v = w_1, \\
dyxy, & a = x, \text{ } b \in \{xy, xyx\} \text{ and } v = w_1, \\
0, & \text{otherwise.}
\end{cases}
\]
Now it is easy to see that $t_1(v t_1 (y \x y \x 1) \x y \x 1) = t_1(v (1 \x r_y \x 1) \x y \x 1)$, so $F_3^{w_1}=0$ on the fourth summand and
\[ F_3^{w_2} (a \x b \x y \x y \x y \x 1 ) = \begin{cases} 
dyxy, & a = y \text{ and } b = yx \\
1+dy, & a = y \text{ and } b = (xy)^2.
\end{cases}
\]
So $F_2^{w_2} =  1$ on the fourth summand in case of $a = y$ and it gives us zero otherwise. Observe that
\begin{itemize}
    \item $F_3^{w_3}\big(a \x x \x y \x y \x y \x 1 \big) = 
et_3\big(dx^2 \x r_y \x 1 + dxyx \x r_y \x x + dx \x r_x \x yx + d^2 x \x r_x \x yxy + d^2 xyx \x r_y \x xy + d^2 xy \x r_x \x (xy)^2)\big)$,
\item $F_3^{w_3}\big(a \x y \x y \x y \x y \x 1 \big) = 
et_3\big(dyx \x r_y \x 1 + dy \x r_x \x yx + d^2 y \x r_x \x yxy + d^2 y^2 \x r_x \x (xy)^2)\big)$,
\item $F_3^{w_3}\big(a \x xy \x y \x y \x y \x 1 \big) = 
et_3\big(dxyx \x r_y \x 1 + dxy \x r_x \x yx + d^2 xy \x r_x \x yxy)\big)$,
\item $F_3^{w_3}\big(a \x yx \x y \x y \x y \x 1 \big) = 
et_3\big(at_2(d(xy)^2 \x r_y \x x + dyx \x r_x \x yx + d^2 yx \x r_x \x yxy + d^2 (yx)^2 \x r_y \x xy + d^2 yxy \x r_x \x (xy)^2)\big)$,
\item  $F_3^{w_3}\big(a \x xyx \x y \x y \x y \x 1 \big) = 
et_3\big(at_2(dxyx \x r_x \x yx + d^2 xyx \x r_x \x yxy + d^2 (xy)^2 \x r_x \x (xy)^2)\big)$,
\item $F_3^{w_3}\big(a \x yxy \x y \x y \x y \x 1 \big) = 
et_3\big(at_2(d(xy)^2 \x r_y \x 1 + dyxy \x r_x \x yx + d^2 yxy \x r_x \x yxy)\big)$,
\item  $F_3^{w_3}\big(a \x (xy)^2 \x y \x y \x y \x 1 \big) = 
et_3\big(at_2(d(xy)^2 \x r_x \x yx + d^2 (xy)^2 \x r_x \x yxy)\big)$,
\end{itemize}
and it is now not hard to prove that $F_3^{w_3}$ gives us zero for any $b \in \{x, y, xy, xyx\}$ after the right multiplying by $(1+dy+d^2xyx)b^*$. So we have
$$F_3^{w_3}\big(a \x b \x y \x y \x y \x (1+dy+d^2 xyx)b^* \big) = \begin{cases}
dxyx + d^2(xy)^2, & a = y \text{ and } b = yx\\
dx, & a = y \text{ and } b = yxy \\
dyx, & a = x \text{ and } b = (xy)^2.
\end{cases}$$
Now consider the seventh summand of $\Phi_5 (1 \x a \x 1)$. Obviously
$$F_3^v(a \x b \x y \x xyx \x y \x 1) = $$
$$=et_3\bigg(at_2\Big( bt_1 \big(v(y \x r_y \x 1 + 1 \x r_y \x y + dy \x r_y \x y + d\x r_y \x xyx + d^2 y\x r_y \x xyx) \x y \x 1\big) \Big) \bigg),$$
and this is not equals zero only for $v = w_3$. In this case it is equal to $d F_3^v(a \x b \x y \x xyx \x y \x 1) = F_3^v(a \x b \x y \x y \x y \x 1) $, so
\[F_3^{w_3}\big(a \x b \x y \x y \x y \x (d+d^2y+d^3 xyx)b^* \big) = F_3^{w_3}\big(a \x b \x y \x y \x y \x (1+dy+d^2 xyx)b^* \big) \]
and $F_3^v = 0$ on eighth and eleventh summands by computations given above. Hence we have
\[ F_3^{w_1} = \begin{cases} 1, & a= x \\ 
0, & a = y,\end{cases} \quad
 F_3^{w_2} = \begin{cases} 0, & a= x \\ 
1, & a = y,\end{cases} \text{  and  }
 F_3^{w_3} = 0.
\]
4) Finally we need to compute $F_4^v$. Since $v t_1 (x \x x \x 1) = v( 1 \x r_x \x 1) \not=0$ only for $v = w_1$, we have $F_4^v = 0$ for the first summand and for any $v \in \{w_2, w_3\}$. In the case of $v = w_1$ we obtain
\[
\sum \limits_{b \in B^*}F_4^{w_1}(a \x b \x x \x x \x x \x b^*)  = et_3 \Big(a t_2 \big(b t_1 (x \x x \x 1) \big) \Big) =
\begin{cases}
1, & a =x \\
0, & a = y
\end{cases}
\]
as described in the cases above. Now it is obvious that $F_4^v$ gives us zero on the eighth summand and
\[\sum \limits_{b \in B^*} F_4^v \big(a \x b \x y \x y \x y \x (1+dy+d^2xyx)b^* \big) = 
\]
\[ =
\begin{cases}
0, & v = w_1 \\
\sum \limits_{b \in B^*} et_3\Big(at_2\big(bt_1(y \x y\x 1)\big)\Big)(1+ dy+d^2xyx)b^*, & v = w_2\\
\sum \limits_{b \in B^*} et_3\Big(at_2\big(bt_1(y \x x\x 1 + dy \x x \x y + dyx \x y \x 1)\big)\Big)(1+ dy+d^2xyx)b^*, & v = w_3\\
\end{cases}
\]
\[
=
\begin{cases}
1, & v = w_2 \text{ and } a = y\\
0, & \text{otherwise.}\\
\end{cases}
\]
So $F_4^v$ gives us zero on the eleventh summand and hence
\[ F_3^{w_1} = \begin{cases} 1, & a= x \\ 
0, & a = y,\end{cases} \quad
 F_3^{w_2} = \begin{cases} 0, & a= x \\ 
1, & a = y, \end{cases} \text{  and  }
 F_3^{w_3} = 0.
\]
According to the computations given above We have proved that
\[\sum \limits_{i=1}^4 F_i^v = \begin{cases}
d yxy, & a = x \text{ and } v = w_1\\
d(xy)^2, & a = x \text{ and } v = w_3 \\
0, & \text{otherwise.}
\end{cases}
\]
Finally, for any $v \in \{w_1, w_2,w_3\}$ we have
\[ 
[v,e] =\sum \limits_{i=1}^2 S_i^v + \sum \limits_{i=1}^4 F_i^v = 0.
\]
\end{proof}

\begin{corollary} Following formulae hold:
\begin{enumerate}
\item $\Delta(q_1e) = dp_1e$ and $\Delta(q_2e) = dp_2e$,
\item $\Delta(ve) = 0$ for any $v \in \{w_1, w_2, w_3\}$.
\end{enumerate}
\end{corollary}
\begin{proof}
Firstly observe that $\Delta(q_1)=dp_1$ and $\Delta(q_2) = \Delta(v)=0$ for any $v \in \{w_1,w_2,w_3\}$ and $\Delta(e)= d^3p_1q_1w_1$ according to Lemmas 2, 3 and 5. So by the Tradler equation 
$\Delta(ab) = \Delta(a)b + a \Delta(b) + [a,b].$
Hence we have
\begin{align*} &\Delta(q_1e) = dp_1e + d^3p_1q_1^2w_1 + 0 = dp_1e + d^3 q_1q_2w_2 = dp_1e,\\
&\Delta(q_2e) = 0 + d^3p_1q_1q_2w_1 + dp_2e = dp_2e,\\
&\Delta(ve) = 0 \cdot e + d^3 p_1q_1 v w_1 + 0 = 0 \end{align*}
for any $v \in \{w_1,w_2,w_3\}$.
\end{proof}

\section{Main Theorem.}
Let $K$ be an algebraically closed field of characteristic 2, let $d \in K$ be a scalar and let $R(2,0,d)$ be an algebra described in item $ 3.1$. So $BV$-structure on Hocschild cohomology algebra $HH^{*}(R)$ can be described in terms of map $\Delta: HH^*(R) \longrightarrow HH^{*}(R)$ of degree $-1$. 
\begin{theorem}
The map $\Delta$ is completely defined by the following equalities: 

\begin{itemize}

\item of degree 1: $\begin{cases}  \Delta(q_1) = dp_1,  \ \Delta(p_1q_1) = p_2+dp_1, \ \Delta(p_1q_2) = dp_2+p_3, \\ 
\Delta(p_2q_1) = p_3+dp_2, \ \Delta(p_3q_2) = p_2, \Delta(p_4q_1) = p_2 , \\
\Delta(p_3q_1) = \Delta(p_2q_2) = p_1, \  \Delta(p_4q_2) = p_3, \end{cases}$
\item of degree 3: $\begin{cases}
\Delta(q_1w_1) = \Delta(q_2w_2) = w_3, \ \Delta(q_2w_1) = q_2^2 + w_2, \\
\Delta(q_1w_2) = \Delta(q_2w_3) = q_1^2 + w_1 + d(p_1+1)w_2, \\
\Delta(q_1w_3) =  q_2^2 + w_2  + dw_3,
\end{cases}$
\item of degree 4: $ \Delta(e) = \Delta(p_4e) = d^3 p_1q_1w_1,$
\item of degree 5: $
\Delta(q_1e) = dp_1e, \ \Delta(q_2e) = dp_2e$,
\item $\Delta(ab) = 0$ for any other combinations of generating elements $a, b \in \mathcal{X} \cup \{1\} $.
\end{itemize}
\end{theorem}
\begin{proof}
For the map defined above $\Delta$ we have equalities of degrees 1, 2, 3 and 4 by Lemmas 2, 3, 4 and 5 respectively and equalities of the higher degrees was given in Corollary 5.
\end{proof}


\begin{thebibliography}{99}
\bibitem{1}  N. Bian, G. Zhang and P. Zhang, \textit{Setwise homotopy category}, Appl. Categ. Structures 17 (2009), no. 6, pp. 561–565.
\bibitem{2} C. Cibils and A. Solotar, \textit{Hochschild cohomology algebra of abelian groups}, Arch. Math. 68 (1997), pp. 17–21.
\bibitem{3} K. Erdmann, \textit{Blocks of tame representation type and related algebras}, Lecture Notes
in Math., v. 1428, Berlin; Heidelberg. 1990.
\bibitem{4} M. Gerstenhaber, \textit{The cohomology structure of an associative ring}, Ann. Math. (2) 78 (1963), pp. 267-288.
\bibitem{5} A. I. Generalov, “Hochschild cohomology of algebras of dihedral type. II. Local algebras”, Zap. Nauchn. Sem. POMI, 375, POMI, St. Petersburg, 2010, 92–129; J. Math. Sci. (N. Y.), 171:3 (2010), 357–379.
\bibitem{6}A. I. Generalov, A. V. Semenov, \textit{Hochschild cohomology of algebras of quaternion type, IV: cohomology algebra for exceptional local algebras}, Zap. Nauchn. Sem. POMI, 478 (2019), pp. 32–77, eng: Journal of Mathematical Sciences, 247 (2019), pp. 518-549.
\bibitem{7} A. I. Generalov, \textit{Hochschild cohomology of algebras of quaternion type, I: Generalized quaternion groups}, Algebra i Analiz, 18:1 (2006), 55–107; St. Petersburg Math. J., 18:1 (2007), 37–76.
\bibitem{8} A. I. Generalov, A. A. Ivanov, S. O. Ivanov, \textit{Hochschild cohomology of algebras of quaternion type. II. The family Q(2B)1 in characteristic 2}, Zap. Nauchn. Sem. POMI, 349, POMI, St. Petersburg, 2007, 53–134; J. Math. Sci. (N. Y.), 151:3 (2008), 2961–3009.
\bibitem{9} G. Hochschild, \textit{On the cohomology groups of an associative algebra}, Ann. Math. (2) 46 (1945), 58-67.
\bibitem{10} A. A. Ivanov, S. O. Ivanov, Yu. Volkov, G. Zhou, \textit{BV structure on Hochschild cohomology of the group ring of the quaternion group of order eight in characteristics two}, Journal of Algebra, 435 (2015), pp. 174-203.
\bibitem{11} A. A. Ivanov, \textit{Hochschild cohomology of algebras of quaternion type. The family $Q(2\mathcal{B})_1$ in characteristic 3}, Zap. Nauchn. Sem. POMI, 388, POMI, St. Petersburg, 2011, 152–178; J. Math. Sci. (N. Y.), 183:5 (2012), 658–674.
\bibitem{12} A. A. Ivanov, \textit{BV-algebra structure on Hochschild cohomology of local algebras of quaternion type in characteristic 2}, Zap. Nauchn. Sem. POMI, 430, 2014, 136–185; J. Math. Sci. (N. Y.), 219:3 (2016), 427–461
\bibitem{13} S. MacLane, \textit{Homology}, Grund. Math. Wiss. 114, Springer-Verlag, Berlin-New York, 1963.
\bibitem{14} L. Meniсhi, \textit{Batalin-Vilkovisky algebras and сyсliс сohomology of Hopf algebras}, K-Theory 32, No. 3 (2004), pp. 231-251
\bibitem{15} L. Meniсhi, \textit{Batalin-Vilkovisky algebra structures on Hochschild cohomology},Bulletin de la Société Mathématique de France, Tome 137 (2009) no. 2, pp. 277-295.
\bibitem{16} T. Tradler, \textit{The Batalin-Vilkovisky algebra on Hochschild cohomology induced by infinity inner products}, Ann. Inst. Fourier 58, No. 7 (2008), pp. 2351-2379.
\bibitem{17} Y. Volkov, \textit{BV-differential on Hochschild cohomology of Frobenius algebras}, Journal of Pure and Applied Algebra, 220, Issue 10, (2016), pp. 3384-3402.
\bibitem{18} Y. Volkov, \textit{The Hochschild cohomology algebra for a family of self-injective algebras of tree class $D_n$},  Algebra i Analiz, 23:5 (2011), 99–139; St Petersburg Mathematical Journal 23 (5), (2012).
\bibitem{19} T. Yang, \textit{A Batalin-Vilkovisky algebra struсture on the Hoсhsсhild сohomology of trunсated polynomials}, arxiv:0707.4213, 2007.
\bibitem{20} Andrei V. Semenov, Alexander Generalov, \textit{$BV$-structure on Hochschild cohomology for exceptional local algebras of quaternion type. Case of even parameter}, preprint, https://arxiv.org/abs/2109.01814
\end{thebibliography}
\end{document}